\documentclass[preprint,12pt]{elsarticle}
\usepackage[a4paper,margin=1in]{geometry}
\usepackage{amssymb}
\usepackage{amsmath}
\usepackage{amsthm}
\usepackage{graphicx}
\usepackage{tikz}
\usetikzlibrary{calc, arrows.meta}
\usepackage{pgfplots}
\usepackage{lineno}
\pgfplotsset{compat=1.18}
\usepackage{hyperref}

\numberwithin{equation}{section}
\newtheorem{theorem}{Theorem}[section]
\newtheorem{corollary}{Corollary}[section]

\newtheorem{lemma}{Lemma}[section]

\newtheorem{proposition}{Proposition}[section]
\theoremstyle{remark}
\newtheorem{remark}{Remark}[section]

\newcommand{\D}{\mathbb{D}}

\begin{document}

\begin{frontmatter}

\ead{rahim.r.kargar@utu.fi, rkargar1983@gmail.com}

\title{Higher-order Volterra-type integral operator on Hardy and Bergman spaces}

\author{Rahim Kargar}
\affiliation{organization={Department of Mathematics and Statistics, University of Turku},%Department and Organization
            city={Turku},
            country={Finland}}

\begin{abstract}

We investigate the higher-order Volterra-type integral operator $T_{g,n}$ in the unit disk $\mathbb D$,
defined for $n\in\mathbb N$ by
\begin{equation*}
T_{g,n}[f](z)
:= \underbrace{\int_{0}^{z}\int_{0}^{t_1}\cdots\int_{0}^{t_{n-1}}}_{n\ \text{times}}
f(t_n)g'(t_n)\,dt_n\cdots dt_1,\quad z\in\mathbb D,
\end{equation*}
where $f$ and $g$ are analytic in $\mathbb D$.
We establish sharp norm and essential norm estimates, and give complete characterizations of boundedness and compactness of $T_{g,n}$ on Hardy spaces $H^p$ and weighted Bergman spaces $A_\alpha^p$, in terms of (vanishing) Carleson measure conditions determined by $|g'|$.

\end{abstract}

\end{frontmatter}

\section{Introduction}

We write $:=$ to indicate that the right-hand side defines the left-hand side.

\subsection{Function spaces} Let $\mathbb{D}:=\{z\in\mathbb{C}: |z|<1\}$ denote the unit disk and let
$\mathbb{T}:=\partial\mathbb{D}=\{z\in\mathbb{C}: |z|=1\}$ be the unit circle.

For $1\le p<\infty$, the Hardy space $H^p$ consists of analytic functions
$f$ in $\mathbb{D}$ with finite norm (see, e.g., \cite[p.~48]{Garnett-book})
\begin{equation}\label{Hardy-norm}
\|f\|_{H^p} := \sup_{0<r<1} M_p(r,f),
\end{equation}
where
\begin{equation}\label{M-p}
M_p(r,f):=\left( \frac{1}{2\pi}\int_0^{2\pi} |f(re^{i\theta})|^p\,d\theta\right)^{1/p}.
\end{equation}
For $f\in H^p$, the function $r\mapsto M_p(r,f)$ is nondecreasing. It follows from the
Minkowski inequality for $p\ge 1$ that $\|\cdot\|_{H^p}$ is a norm, and $H^p$ is complete
with respect to this norm; hence $H^p$ is a Banach space. If $p<q$, then
$H^q\subset H^p$.

\medskip
Given a positive Borel measure $\mu$ on $\mathbb D$ and $0<p<\infty$, we write
$L^p(\mathbb D,\mu)$ for the space of Borel measurable functions
$f:\mathbb D\to\mathbb C$ such that
\begin{equation*}
\|f\|_{L^p(\mathbb D,\mu)}^p:=\int_{\mathbb D}|f(z)|^p\,d\mu(z)<\infty.
\end{equation*}
When $f$ is analytic (in particular, $f\in H^p$), this integral is well-defined since analytic functions are measurable.

For an analytic function $f$ in $\mathbb{D}$ and a point $a\in\mathbb{D}$,
the hyperbolic translate of $f$ is defined by
\begin{equation*}
f_a(z):=f\!\left(\frac{z+a}{1+\overline a z}\right)-f(a).
\end{equation*}
The space $\mathrm{BMOA}$ consists of those functions $f\in H^2$ such that
\begin{equation*}
\|f\|_*:=|f(0)|+\sup_{a\in\mathbb{D}}\|f_a\|_{H^2}<\infty.
\end{equation*}
It is well-known that $\mathrm{BMOA}\subset H^p$ for every $1\le p<\infty$.
The subspace $\mathrm{VMOA}$ is defined as the closure of the polynomials in
$\mathrm{BMOA}$, or equivalently, as the collection of functions $f\in\mathrm{BMOA}$
satisfying $\|f_a\|_{H^2}\to0$ as $|a|\to1^{-}$.

Recall that a positive Borel measure $\mu$ on $\mathbb D$ is a Carleson measure
(for $H^p$, $1\le p<\infty$) if there exists a constant $C>0$ such that
\begin{equation*}
\mu(S(I))\le C\,|I|
\end{equation*}
for every arc $I\subset\mathbb T$, where $|I|:=m(I)/(2\pi)$ denotes the normalized arc length
($m$ is the Lebesgue measure on $\mathbb T$), and
\begin{equation*}
S(I):=\{\,re^{i\theta}:\ e^{i\theta}\in I,\ 1-|I|\le r<1\,\}
\end{equation*}
is the associated Carleson box \cite[Ch. II]{Garnett-book}. For more illustration, see Figure \ref{fig:carleson-box} below.
We say that $\mu$ is a vanishing Carleson measure if
\begin{equation*}
\frac{\mu(S(I))}{|I|}\longrightarrow 0 \quad \text{as } |I|\to 0,
\end{equation*}
uniformly with respect to the position of $I$ on $\mathbb T$ (equivalently, the supremum over all arcs $I$
with $|I|\le \delta$ tends to $0$ as $\delta\to0$); see \cite{Duren1970,Pommerenke1977}. On the other hand, it is well known that $\mu$ is a Carleson measure if and only if
\begin{equation}\label{Carleson-property}
\int_{\mathbb{D}} |f(z)|^p\,d\mu(z)\le C_p\,\|f\|_{H^p}^p,
\quad 0<p<\infty,
\end{equation}
for all $f\in H^p$; see \cite[p.~231]{Garnett-book} or \cite[Theorem 9.3]{Duren1970}.

In terms of Carleson measures, for an analytic function $f$ in $\mathbb{D}$,
the condition $f\in\mathrm{BMOA}$ is equivalent to requiring that the measure
\begin{equation*}
|f'(z)|^2(1-|z|^2)\,dA(z)
\end{equation*}
is a Carleson measure on $\mathbb{D}$, where $dA(z) = ({1}/{\pi})dxdy$ is the normalized area measure on $\mathbb{D}$.
Similarly, $f\in\mathrm{VMOA}$ if and only if this measure is a vanishing Carleson measure; see \cite{Aleman} and \cite[p.~262]{Garnett-book}.

Let $\mathcal{B}$ denote the Bloch space of analytic functions $f$ in $\mathbb{D}$ such that
\begin{equation*}
\|f\|_{\mathcal{B}}:=|f(0)|+\sup_{z\in\mathbb{D}} (1-|z|^2)|f'(z)|<\infty.
\end{equation*}
The closure of the polynomials in $\|\cdot\|_{\mathcal{B}}$ is called the little Bloch space,
which is denoted by $\mathcal{B}_0$. We have $\mathrm{BMOA}\subset \mathcal{B}$ and
$\mathrm{VMOA}\subset \mathcal{B}_0$.

For $\alpha>-1$ and $1\le p<\infty$, the weighted Bergman space $A_\alpha^p$ consists of
analytic functions $f$ in $\mathbb D$ such that (see, e.g., \cite[p.~2]{Hedenmalm})
\begin{equation}\label{eq:bergman-norm}
\|f\|_{A_\alpha^p}^p:=\int_{\mathbb D}|f(z)|^p\,dA_\alpha(z)<\infty,
\quad
dA_\alpha(z):=(\alpha+1)(1-|z|^2)^\alpha\,dA(z).
\end{equation}

A positive Borel measure $\mu$ on $\mathbb{D}$ is called a $(p,q)$-Carleson measure
for $A_\alpha^p$ if there exists a constant $C > 0$ such that
\begin{equation*}
\int_{\mathbb{D}} |f(z)|^q \, d\mu(z) \le C \|f\|_{A_\alpha^p}^q
\end{equation*}
for all $f \in A_\alpha^p$. We say that $\mu$ is a vanishing $(p,q)$-Carleson measure
if for every $\varepsilon > 0$, there exists $\delta > 0$ such that
\begin{equation*}
\int_{S(I)} |f(z)|^q \, d\mu(z) \le \varepsilon \|f\|_{A_\alpha^p}^q
\end{equation*}
for all arcs $I$ with $|I| < \delta$ and all $f \in A_\alpha^p$ with $\|f\|_{A_\alpha^p} \le 1$, see \cite{oscar, Dya}.

For the theory of Carleson measures on Bergman spaces, we refer to
\cite{Hastings1975, Luecking1986, Zhu2007}.

\subsection{The higher-order Volterra-type integral operator} The classical Volterra-type integral operator with symbol $g$, introduced by Pommerenke \cite{Pommerenke1977}, is defined by \begin{equation*} T_g[f](z):=\int_0^z f(t)g'(t)\,dt,\quad z\in\mathbb{D}, \end{equation*} where $f$ and $g$ are analytic in $\mathbb{D}$. Pommerenke proved that $T_g$ is bounded on $H^2$ if and only if $g\in\mathrm{BMOA}$. The fundamental work of Aleman and Siskakis later established that, for $1\le p<\infty$, the operator $T_g$ is bounded on $H^p$ if and only if $g\in\mathrm{BMOA}$, and compact on $H^p$ if and only if $g\in\mathrm{VMOA}$ \cite{Aleman}. These results have generated extensive research on Volterra-type operators on Hardy, Bergman, and related spaces; see, for instance, \cite{Aleman-Cima-2001, AlemanSiskakis1997, Niko2020, Niko2024, zhou-zhu} and the references therein.

In this paper, we study a higher-order Volterra-type operator obtained by iterating the integration procedure. For $n\in\mathbb{N}$ and analytic functions $f$ and $g$ in $\mathbb{D}$, we define \begin{equation}\label{eq:def-Tgn} T_{g,n}[f](z) :=\underbrace{\int_0^z\int_0^{t_1}\cdots\int_0^{t_{n-1}}}_{n\ \text{times}} f(t_n)g'(t_n)\,dt_n\cdots dt_1, \quad z\in\mathbb{D}. \end{equation} Here, each integral is taken along the line segment from $0$ to the upper limit; equivalently, since the integrand is analytic, the value is independent of the choice of rectifiable path in $\mathbb{D}$. This operator admits the equivalent convolution representation \begin{equation}\label{eq:conv-repr} T_{g,n}[f](z)=\frac{1}{(n-1)!}\int_0^z (z-t)^{\,n-1}f(t)g'(t)\,dt, \quad z\in\mathbb{D}, \end{equation} and clearly reduces to $T_g$ when $n=1$. This generalization was introduced in \cite{Kar-Volterra}.

A useful structural feature of $T_{g,n}$ is its factorization through the classical integration operator. Let \begin{equation*} V\varphi(z):=\int_0^z \varphi(t)\,dt \end{equation*} and denote by $V^k$ the $k$-fold integration operator \begin{equation}\label{V-k} V^k\varphi(z)=\underbrace{\int_0^z\int_0^{t_1}\cdots\int_0^{t_{k-1}}}_{k\ \text{times}} \varphi(t_k)\,dt_k\cdots dt_1. \end{equation} Using the same method as in \cite[Proposition 4.1]{Kar-Volterra}, one obtains the convolution formula \begin{equation}\label{eq:Vk-conv} V^k\varphi(z)=\frac{1}{(k-1)!}\int_0^z (z-t)^{k-1}\varphi(t)\,dt. \end{equation} If we use the iterated integral representation \eqref{eq:def-Tgn} (or by induction on $n$), then we verify the factorization \begin{equation}\label{eq:factorization} T_{g,n}=V^{\,n-1}\circ T_{g,1}. \end{equation} Another useful representation of $T_{g,n}$ involves reduced operators. For $0\le k\le n-1$, define \begin{equation}\label{Ak} A_{k,g}[f](z):=\int_0^z t^{k} f(t)g'(t)\,dt,\quad z\in\mathbb{D}. \end{equation} Using the binomial expansion of $(z-t)^{n-1}$ in the convolution representation \eqref{eq:conv-repr}, one can express $T_{g,n}$ as a finite linear combination \begin{equation}\label{eq:Tn-reduced} T_{g,n}[f](z)=\frac{1}{(n-1)!}\sum_{k=0}^{n-1} (-1)^k\binom{n-1}{k} z^{n-1-k}\,A_{k,g}[f](z). \end{equation} This decomposition is established in Lemma~\ref{lem:fin dec} below and provides an alternative approach to studying $T_{g,n}$ through the simpler operators $A_{k,g}$.

%***************************************************
\section{Boundedness and Compactness}
%***************************

As a first step, we show that the operator $T_{g,n}$ is bounded. For this, we need the following lemma on the
$k$-fold integration operator.

\begin{lemma}\label{lem:Vk-sharp}
Let $1\le p<\infty$ and $k\in\mathbb{N}$. Then the $k$-fold integration operator
$V^k:H^p\to H^p$, defined by \eqref{V-k}, is bounded and
\begin{equation*}
\|V^k\|_{H^p\to H^p}=\frac{1}{k!}.
\end{equation*}
\end{lemma}

\begin{proof}
Let $\varphi\in H^p$ and $0<r<1$. Using the convolution representation \eqref{eq:Vk-conv} and parametrizing the radial segment from $0$ to $z=re^{i\theta}$ by $t=se^{i\theta}$, $0\le s\le r$, we obtain
\begin{align*}
V^k\varphi(re^{i\theta})
&= \frac{1}{(k-1)!}\int_{0}^{re^{i\theta}}(re^{i\theta}-t)^{k-1}\varphi(t)\,dt \\
&= \frac{1}{(k-1)!}\int_{0}^{r}(re^{i\theta}-se^{i\theta})^{k-1}\varphi(se^{i\theta})\,e^{i\theta}\,ds \\
&= \frac{e^{ik\theta}}{(k-1)!}\int_{0}^{r}(r-s)^{k-1}\varphi(se^{i\theta})\,ds.
\end{align*}
Taking absolute values and applying Minkowski's integral inequality (valid since $p\ge1$), we obtain
\begin{align*}
M_p(r,V^k\varphi)
&= \left(\frac{1}{2\pi}\int_{0}^{2\pi}\bigl|V^k\varphi(re^{i\theta})\bigr|^p\,d\theta\right)^{1/p} \\
&\le \frac{1}{(k-1)!}\int_{0}^{r}(r-s)^{k-1}
\left(\frac{1}{2\pi}\int_{0}^{2\pi}\bigl|\varphi(se^{i\theta})\bigr|^p\,d\theta\right)^{1/p} ds \\
&= \frac{1}{(k-1)!}\int_{0}^{r}(r-s)^{k-1} M_p(s,\varphi)\,ds.
\end{align*}
Since $M_p(s,\varphi)$ is nondecreasing in $s$ and bounded by $\|\varphi\|_{H^p}$ for $0<s<1$, it follows that
\begin{align*}
M_p(r,V^k\varphi)
&\le \frac{\|\varphi\|_{H^p}}{(k-1)!}\int_{0}^{r}(r-s)^{k-1}\,ds
= \frac{r^k}{k!}\,\|\varphi\|_{H^p}
\le \frac{1}{k!}\,\|\varphi\|_{H^p}.
\end{align*}
Taking the supremum over $0<r<1$ yields
\begin{equation*}
\|V^k\varphi\|_{H^p}\le \frac{1}{k!}\,\|\varphi\|_{H^p},
\quad\text{hence}\quad
\|V^k\|_{H^p\to H^p}\le \frac{1}{k!}.
\end{equation*}

For the reverse inequality, consider $\varphi\equiv1$. Then $\|\varphi\|_{H^p}=1$ and
\begin{equation*}
V^k1(z)=\frac{1}{(k-1)!}\int_{0}^{z}(z-t)^{k-1}\,dt=\frac{z^k}{k!}.
\end{equation*}
Since $\|z^k\|_{H^p}=1$ for all $k\ge0$, we obtain
\begin{equation*}
\|V^k\|_{H^p\to H^p}\ge \|V^k1\|_{H^p}=\frac{1}{k!}.
\end{equation*}
Combining the estimates gives the desired equality.
\end{proof}

We now estimate the norm of the operator $T_{g,n}$.

\begin{lemma}\label{lem:norm-est}
Let $1\le p<\infty$ and $n\in\mathbb{N}$. Let $g$ be analytic in $\mathbb{D}$ such that $T_{g,1}:H^p\to H^p$ is bounded.
Then
\begin{equation}\label{eq:Tgn-norm}
\|T_{g,n}\|_{H^p\to H^p}
\le C_n\,\|T_{g,1}\|_{H^p\to H^p},
\quad\text{where}\quad
C_n:=\frac{1}{(n-1)!}.
\end{equation}
Moreover, the constant $C_n=1/(n-1)!$ is optimal among all bounds obtained via the factorization \eqref{eq:factorization} and submultiplicativity of operator norms.
\end{lemma}

\begin{proof}
By the factorization \eqref{eq:factorization}, namely $T_{g,n}=V^{\,n-1}\circ T_{g,1}$,
and the submultiplicativity of operator norms, we obtain
\begin{equation*}
\|T_{g,n}\|_{H^p\to H^p}
=\|V^{\,n-1}\circ T_{g,1}\|_{H^p\to H^p}
\le \|V^{\,n-1}\|_{H^p\to H^p}\,\|T_{g,1}\|_{H^p\to H^p}.
\end{equation*}
By Lemma~\ref{lem:Vk-sharp}, $\|V^{\,n-1}\|_{H^p\to H^p}=1/(n-1)!$, and therefore
\begin{equation*}
\|T_{g,n}\|_{H^p\to H^p}
\le \frac{1}{(n-1)!}\,\|T_{g,1}\|_{H^p\to H^p},
\end{equation*}
which proves \eqref{eq:Tgn-norm}.

Since the estimate follows from the inequality
\begin{equation*}
\|T_{g,n}\|\le \|V^{n-1}\|\cdot\|T_{g,1}\|,
\end{equation*}
and the factor $\|V^{n-1}\|=1/(n-1)!$ is sharp, this argument can obtain no smaller constant.
\end{proof}

\begin{corollary}\label{cor:BMOA-norm}
Let $1\le p<\infty$ and $n\in\mathbb{N}$. If $g\in\mathrm{BMOA}$, then
\begin{equation*}
\|T_{g,n}\|_{H^p\to H^p}
\lesssim_{n}
\frac{1}{(n-1)!}\,\|g\|_{\mathrm{BMOA}}.
\end{equation*}
\end{corollary}

\begin{proof}
If $g\in\mathrm{BMOA}$, then by the Aleman--Siskakis theorem \cite{Aleman},
\begin{equation*}
\|T_{g,1}\|_{H^p\to H^p}\lesssim_p \|g\|_{\mathrm{BMOA}}.
\end{equation*}
Combining this with Lemma~\ref{lem:norm-est} yields the desired estimate.
\end{proof}

\begin{remark}\label{rem:sufficient-only}
Corollary~\ref{cor:BMOA-norm} provides a sufficient condition for the boundedness
of $T_{g,n}$ on $H^p$. For $n\ge 2$, this condition is not expected to be necessary in general; more refined symbol characterizations can be obtained by viewing $T_{g,n}$ as a special case of generalized integration operators.
\end{remark}

\begin{proposition}\label{prop:diff-id}
For all analytic functions $f,g$ in $\mathbb{D}$ and every $n\in\mathbb{N}$,
\begin{equation}\label{deff nth T}
\frac{d^{\,n}}{dz^{\,n}} T_{g,n}[f](z) = f(z)\, g'(z),\quad z\in\mathbb{D}.
\end{equation}
\end{proposition}

\begin{proof}
Let $h(t):=f(t)g'(t)$. Using the convolution representation \eqref{eq:conv-repr},
\begin{equation*}
T_{g,n}[f](z)=\frac{1}{(n-1)!}\int_0^z (z-t)^{n-1}h(t)\,dt.
\end{equation*}
Differentiate under the integral sign (Leibniz's rule). Since $(z-t)^{n-1}$ vanishes at $t=z$ for $n\ge2$, no boundary term appears when differentiating up to order $n-1$. By repeated differentiation under the integral sign, we obtain for $1\le k\le n-1$,
\begin{equation*}
\frac{d^k}{dz^k}\int_0^z (z-t)^{n-1}h(t)\,dt
= \frac{(n-1)!}{(n-1-k)!}\int_0^z (z-t)^{n-1-k}h(t)\,dt.
\end{equation*}
Taking $k=n-1$ gives
\begin{equation*}
\frac{d^{\,n-1}}{dz^{\,n-1}}T_{g,n}[f](z)
= \frac{1}{(n-1)!}\cdot (n-1)!\int_0^z h(t)\,dt
= \int_0^z h(t)\,dt.
\end{equation*}
Differentiating once more proves \eqref{deff nth T}.
\end{proof}

We are now ready to prove the first main result.

\begin{theorem}\label{thm:Hp-Tgn-Carleson}
Let $1\le p<\infty$ and $n\in\mathbb N$. Let $g$ be analytic in $\mathbb{D}$ and define $T_{g,n}$ by \eqref{eq:def-Tgn}. Then the operator $T_{g,n}:H^p\to H^p$ is bounded if and only if the measure
\begin{equation}\label{dmu bnp}
d\mu_{g,n,p}(z)=|g'(z)|^p(1-|z|^2)^{np-1}\,dA(z)
\end{equation}
is a Carleson measure for $H^p$.

Moreover, $T_{g,n}$ is compact on $H^p$ if and only if $\mu_{g,n,p}$ is a vanishing Carleson measure.
\end{theorem}

\begin{proof}
\textit{Boundedness.}
Assume that $ \mu_{g,n,p} $ defined in \eqref{dmu bnp} is a Carleson measure for $ H^p $. We show that $ T_{g,n} $ is bounded on $ H^p $.

First, we recall the Hardy-Sobolev norm equivalence for analytic functions in $ H^p $. For any analytic function $ F $ on the unit disk $ \mathbb{D} $ and for every integer $ n \geq 1 $, the following equivalence holds:
\begin{equation}\label{eq:LP}
\|F\|_{H^p}^p \asymp_{p,n} |F(0)|^p + \int_{\mathbb{D}} |F^{(n)}(z)|^p (1 - |z|^2)^{np - 1} \, dA(z),
\end{equation}
which is a standard result for Hardy spaces (see \cite[Proposition 6.2]{Zhu2005}). Since $T_{g,n}=0$, the term $|F(0)|^p$ vanishes, and thus \eqref{eq:LP} yields
\begin{align*}
\|T_{g,n}[f]\|_{H^p}^p
&\asymp_{p,n}
\int_{\mathbb D} |(T_{g,n}[f])^{(n)}(z)|^p (1-|z|^2)^{np-1}\,dA(z)\\
&=
\int_{\mathbb D} |f(z)|^p |g'(z)|^p (1-|z|^2)^{np-1}\,dA(z)
=\int_{\mathbb D}|f(z)|^p\,d\mu_{g,n,p}(z).
\end{align*}
By the Carleson embedding theorem \eqref{Carleson-property},
\begin{equation*}
\int_{\mathbb D}|f(z)|^p\,d\mu_{g,n,p}(z)\lesssim \|f\|_{H^p}^p,
\end{equation*}
and therefore $\|T_{g,n}[f]\|_{H^p}\lesssim \|f\|_{H^p}$, proving boundedness.

Conversely, assume that $T_{g,n}:H^p\to H^p$ is bounded. Let $I\subset\mathbb T$ be an arc and $S(I)$ its Carleson box.
Choose $a\in\mathbb D$ such that $1-|a|\asymp |I|$ and $a/|a|$ lies above the midpoint of $I$, see Figure \ref{fig:carleson-box}.

\begin{figure}
    \centering
\begin{tikzpicture}[scale=4.0, line cap=round, line join=round]
  \def\thetaA{25}      % left endpoint angle of I (degrees)
  \def\thetaB{55}      % right endpoint angle of I (degrees)
  \def\thetaM{40}      % midpoint angle of I (degrees)
  \def\h{0.30}         % |I| ~ h (schematic), Carleson box height ~ h
  \def\r{0.70}         % choose so that 1-|a| \asymp |I| (here 1-\r = 0.30)

  % Unit disk
  \draw[thick] (0,0) circle (1);
  \fill (0,0) circle (0.015);
  \node at (-0.12,-0.10) {$0$};

  % Arc I on the unit circle
  \draw[very thick, blue] ({cos(\thetaA)},{sin(\thetaA)}) arc (\thetaA:\thetaB:1);
  \node at (1,0.6) {$I\subset\mathbb T$};

  % Inner circle r = 1-h (schematic)
  \draw[dashed] (0,0) circle ({1-\h});

  % Carleson box boundary (radial sides)
  \draw[thick]
    ({(1-\h)*cos(\thetaA)},{(1-\h)*sin(\thetaA)}) -- ({cos(\thetaA)},{sin(\thetaA)});
  \draw[thick]
    ({(1-\h)*cos(\thetaB)},{(1-\h)*sin(\thetaB)}) -- ({cos(\thetaB)},{sin(\thetaB)});

  % Carleson box lower boundary (on r=1-h)
  \draw[thick]
    ({(1-\h)*cos(\thetaA)},{(1-\h)*sin(\thetaA)})
    arc (\thetaA:\thetaB:{1-\h});

  % Label S(I)
  \node at (0.59,0.64) {$S(I)$};

  % Midpoint ray
  \draw[dotted] (0,0) -- ({1.02*cos(\thetaM)},{1.02*sin(\thetaM)});

  % Points a and a/|a|
  \coordinate (ptA) at ({\r*cos(\thetaM)},{\r*sin(\thetaM)});
  \coordinate (ptU) at ({cos(\thetaM)},{sin(\thetaM)});

  \fill (ptA) circle (0.02);
  \node at (0.5,0.35) {$a$};

  \fill (ptU) circle (0.02);
  \node at (0.8,0.76) {$\frac{a}{|a|}$};

  \node at (-0.9,-0.8) {$\mathbb D$};

  % Distance marker
  \draw[<->] (ptA) -- (ptU);

  \coordinate (midSeg) at ($(ptA)!0.5!(ptU)$);

  \node[fill=white, inner sep=2pt] (lab) at (1.15,-0.25) {$1-|a|\asymp |I|$};
  \draw[-{Stealth[length=2mm]}] (lab.west) -- (midSeg);
\end{tikzpicture}
   \caption{A Carleson box $S(I)$ associated with an arc $I\subset\mathbb T$ (blue) and a point
$a\in\mathbb D$ chosen above the midpoint of $I$, with $1-|a|\asymp |I|$.}
    \label{fig:carleson-box}
\end{figure}
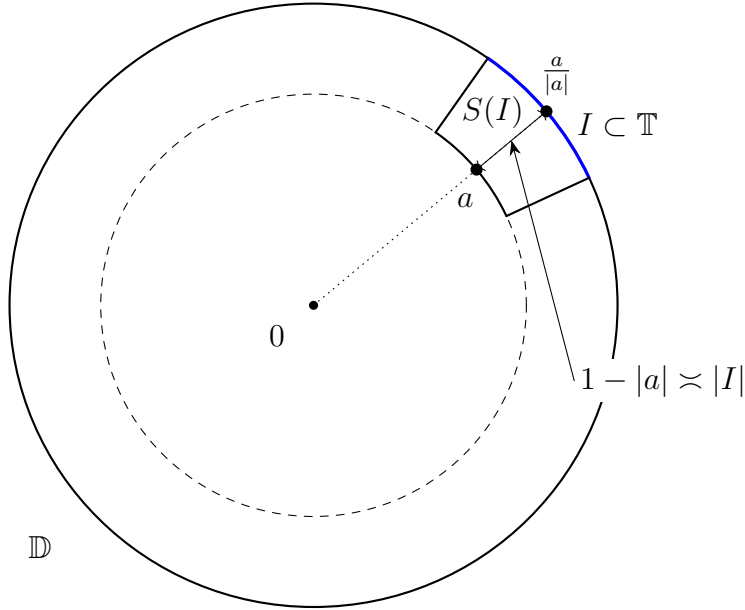

Consider the standard normalized test function
\begin{equation*}
k_a(z):=\left(\frac{1-|a|^2}{(1-\overline a z)^2}\right)^{1/p}.
\end{equation*}
Then $\|k_a\|_{H^p}\asymp 1$, and for $z\in S(I)$ one has $|1-\overline a z|\asymp 1-|a|\asymp |I|$. Hence
\begin{equation}\label{eq:ka-on-SI}
|k_a(z)|^p=\frac{1-|a|^2}{|1-\overline a z|^2}\asymp \frac{1}{|I|}, \quad z\in S(I),
\end{equation}
see \cite[Ch.~2]{Garnett-book}. By boundedness,
\begin{equation*}
\|T_{g,n}[k_a]\|_{H^p}^p\lesssim \|k_a\|_{H^p}^p\asymp 1.
\end{equation*}
Using \eqref{eq:LP} and Proposition~\ref{prop:diff-id} again gives
\begin{align*}
\|T_{g,n}[k_a]\|_{H^p}^p
&\asymp_{p,n}
\int_{\mathbb D}|k_a(z)|^p |g'(z)|^p (1-|z|^2)^{np-1}\,dA(z)\\
&\ge
\int_{S(I)}|k_a(z)|^p |g'(z)|^p (1-|z|^2)^{np-1}\,dA(z).
\end{align*}
For $z\in S(I)$ we also have $1-|z|^2\asymp |I|$, hence by \eqref{eq:ka-on-SI},
\begin{align*}
1
&\gtrsim \|T_{g,n}[k_a]\|_{H^p}^p\\
&\gtrsim
\int_{S(I)} |k_a(z)|^p |g'(z)|^p (1-|z|^2)^{np-1}\,dA(z)
\asymp
\frac{1}{|I|}\int_{S(I)} |g'(z)|^p (1-|z|^2)^{np-1}\,dA(z)\\
&=
\frac{\mu_{g,n,p}(S(I))}{|I|}.
\end{align*}
Thus, $\mu_{g,n,p}(S(I))\lesssim |I|$ uniformly in $I$, so $\mu_{g,n,p}$ is a Carleson measure.

\medskip
\noindent\textit{Compactness.}
Assume first that $\mu_{g,n,p}$ is a vanishing Carleson measure. It is a standard fact that
vanishing Carleson measures are precisely those measures for which the embedding
\begin{equation*}
H^p \hookrightarrow L^p(\mathbb D,\mu)
\end{equation*}
is compact (see, \cite[Ch.~3]{Garnett-book}).
Thus, if $\{f_k\}$ is bounded in $H^p$ and $f_k\to 0$ uniformly on compact subsets of $\mathbb D$, then
\begin{equation}\label{eq:compact-embed}
\int_{\mathbb D} |f_k(z)|^p\,d\mu_{g,n,p}(z)\longrightarrow 0.
\end{equation}
Using \eqref{eq:LP} and Proposition~\ref{prop:diff-id} as before, we obtain
\begin{equation*}
\|T_{g,n}[f_k]\|_{H^p}^p
\asymp_{p,n}
\int_{\mathbb D} |f_k(z)|^p\,d\mu_{g,n,p}(z)\longrightarrow 0,
\end{equation*}
which shows that $T_{g,n}$ is compact on $H^p$.

Conversely, suppose that $T_{g,n}$ is compact on $H^p$.
Let $\{a_j\}\subset \mathbb D$ with $|a_j|\to 1$. Then $k_{a_j}\to 0$ uniformly on compact subsets of $\mathbb D$
and $\|k_{a_j}\|_{H^p}\asymp 1$, hence compactness implies
\begin{equation*}
\|T_{g,n}[k_{a_j}]\|_{H^p}\longrightarrow 0 \quad \text{as }j\to\infty.
\end{equation*}
Let $I_j$ be the arc associated with $a_j$ as in the boundedness part, so that $1-|a_j|\asymp |I_j|$ and $S(I_j)$
is the corresponding Carleson box. Repeating the lower bound estimate above yields
\begin{equation*}
\|T_{g,n}[k_{a_j}]\|_{H^p}^p
\gtrsim \frac{\mu_{g,n,p}(S(I_j))}{|I_j|}.
\end{equation*}
Therefore, $\mu_{g,n,p}(S(I_j))/|I_j|\to 0$ as $|I_j|\to 0$. Since every Carleson box arises (up to uniform comparability) from such a choice of $a_j$, this proves the vanishing Carleson condition. Hence, $\mu_{g,n,p}$ is a vanishing Carleson measure.
\end{proof}

%===============================================

\begin{remark}\label{rem:BMOA_fails_nge2}
For $n=1$ one has the classical Aleman--Siskakis characterizations
\begin{equation*}
T_{g,1}\ \text{bounded on }H^p \iff g\in \mathrm{BMOA},
\quad\text{and}\quad
T_{g,1}\ \text{compact on }H^p \iff g\in \mathrm{VMOA},
\end{equation*}
where $1\le p<\infty$.
However, for $n\ge2$, the additional smoothing produced by iterated
integration changes the behavior of $T_{g,n}$, and boundedness
(compactness) can no longer be captured solely by membership of the
symbol $g$ in $\mathrm{BMOA}$ ($\mathrm{VMOA}$). As a concrete
illustration, consider the symbol defined by
\begin{equation*}
g'(z)=\frac{1}{(1-z)^{3/2}},\quad z\in\mathbb{D},
\end{equation*}
and let $g$ be its primitive with $g(0)=0$. We claim that
$g\notin\mathrm{BMOA}$ but $T_{g,2}$ is bounded on $H^2$.

\smallskip
\noindent\textit{i) $g\notin\mathrm{BMOA}$.}
Recall that $g\in \mathrm{BMOA}$ if and only if the measure
\begin{equation*}
d\mu(z)=|g'(z)|^2(1-|z|^2)\,dA(z)
= \frac{1-|z|^2}{|1-z|^{3}}\,dA(z)
\end{equation*}
is a Carleson measure for $H^2$.
Let $I_\ell=\{e^{i\theta}:|\theta|<\ell\}$ and $S(I_\ell)$ be the
associated Carleson box. For $z=re^{i\theta}\in S(I_\ell)$, one has
\begin{equation*}
|1-z|^2 \asymp (1-r)^2+\theta^2,\qquad 1-|z|^2\asymp 1-r,
\end{equation*}
so
\begin{equation*}
d\mu(z)\asymp \frac{1-r}{\bigl((1-r)^2+\theta^2\bigr)^{3/2}}\,dA(z).
\end{equation*}
Setting $t=1-r$ and integrating over $S(I_\ell)$,
\begin{equation*}
\mu(S(I_\ell))
\asymp \int_0^\ell \int_{-\ell}^{\ell}
\frac{t}{(t^2+\theta^2)^{3/2}}\, d\theta\, dt.
\end{equation*}
For the inner integral, the substitution $\theta = ts$ gives
\begin{equation*}
\int_{-\ell}^{\ell} \frac{t}{(t^2+\theta^2)^{3/2}}\,d\theta
= \frac{1}{t}\int_{-\ell/t}^{\ell/t}\frac{ds}{(1+s^2)^{3/2}}
\asymp \frac{1}{t},
\end{equation*}
where the last step uses the fact that
$\int_{-\infty}^{\infty}\frac{ds}{(1+s^2)^{3/2}}=2<\infty$
and $\ell/t\to\infty$ as $t\to 0^+$. Hence
\begin{equation*}
\mu(S(I_\ell))
\asymp \int_0^\ell \frac{dt}{t},
\end{equation*}
which diverges as $t\to 0^+$. Thus, $\mu$ is not a Carleson measure
and $g\notin \mathrm{BMOA}$.

\smallskip
\noindent\textit{ii) $T_{g,2}$ is bounded on $H^2$.}
By Theorem~\ref{thm:Hp-Tgn-Carleson} with $(n,p)=(2,2)$, $T_{g,2}$
is bounded if and only if
\begin{equation*}
d\mu_{g,2,2}(z)=|g'(z)|^2(1-|z|^2)^3\,dA(z)
=\frac{(1-|z|^2)^3}{|1-z|^{3}}\,dA(z)
\end{equation*}
is a Carleson measure. Using $|1-z|\ge 1-|z|$ and
$1-|z|^2\le 2(1-|z|)$, we obtain
\begin{equation*}
\frac{(1-|z|^2)^3}{|1-z|^{3}}
\le \frac{8(1-|z|)^3}{(1-|z|)^3}
=8.
\end{equation*}
Hence, for every arc $I\subseteq\partial\mathbb{D}$,
\begin{equation*}
\mu_{g,2,2}(S(I))
\le 8\,\mathrm{Area}(S(I))
\asymp |I|^2
\le 2\pi\,|I|,
\end{equation*}
where the last step uses $|I|\le 2\pi$. Thus, $\mu_{g,2,2}$ is a
Carleson measure and $T_{g,2}$ is bounded on $H^2$.

More generally, consider the family
\begin{equation*}
g'(z) = (1-z)^{-\alpha}, \quad z\in\mathbb{D},
\end{equation*}
with $1<\alpha\le 3/2$.
A computation analogous to part~(i), using
\begin{equation*}
|1-z|^2 \asymp (1-r)^2+\theta^2,
\end{equation*}
shows that
\begin{equation*}
\mu(S(I_\ell))
\asymp \int_0^\ell \int_{-\ell}^{\ell}
\frac{t}{(t^2+\theta^2)^{\alpha}}\, d\theta\, dt.
\end{equation*}
Again, by the same argument as analogous to part~(i), we get
\begin{equation*}
\int_{-\ell}^{\ell} \frac{t}{(t^2+\theta^2)^{\alpha}}\, d\theta
\asymp t^{1-2\alpha},
\end{equation*}
so that
\begin{equation*}
\mu(S(I_\ell))
\asymp \int_0^\ell t^{1-2\alpha}\, dt,
\end{equation*}
which diverges for all $\alpha\ge 1$.

Therefore, $g\notin\mathrm{BMOA}$ for all $\alpha\geq 1$.
On the other hand, for $\alpha\le 3/2$ the density of $\mu_{g,2,2}$
satisfies
\begin{equation*}
\frac{(1-|z|^2)^3}{|1-z|^{2\alpha}}
\asymp
\frac{(1-|z|)^3}{|1-z|^{2\alpha}}
\leq C,
\end{equation*}
since $2\alpha\leq 3$, so the pointwise bound gives the Carleson
condition by the same argument as above.
Consequently, $T_{g,2}$ is bounded on $H^2$ for all
$\alpha\in(1,3/2]$, while $g\notin\mathrm{BMOA}$.
This shows that the $\mathrm{BMOA}$ criterion for $T_{g,1}$ does
not extend to the higher-order operator $T_{g,2}$.
\end{remark}
%===============================================

Next, we present an alternative method to study $T_{g,n}$ using the simpler operators $A_{k,g}$,
which will be useful for the second main result.

%===============================================

\begin{lemma}\label{lem:fin dec}
Let $n\in\mathbb{N}$ and let $f,g$ be analytic on $\mathbb{D}$. Then, for every
$z\in\mathbb{D}$,
\begin{equation}\label{eq:Tn_linear_comb}
T_{g,n}[f](z)
=\frac{1}{(n-1)!}\sum_{k=0}^{n-1}
(-1)^k\binom{n-1}{k}\,z^{\,n-1-k}\,A_{k,g}[f](z),
\end{equation}
where $A_{k,g}$ is defined by \eqref{Ak}.
\end{lemma}

\begin{proof}
Fix $z\in\mathbb{D}$. By the binomial theorem,
\begin{equation}\label{eq:binomial_expansion}
(z-t)^{n-1}=\sum_{k=0}^{n-1}(-1)^k\binom{n-1}{k} z^{\,n-1-k} t^k,
\quad t\in\mathbb{C}.
\end{equation}
Substituting \eqref{eq:binomial_expansion} into the convolution representation
\eqref{eq:conv-repr} yields
\begin{align*}
T_{g,n}[f](z)
&=\frac{1}{(n-1)!}\int_0^z (z-t)^{n-1} f(t)g'(t)\,dt\\
&=\frac{1}{(n-1)!}\int_0^z
\left(\sum_{k=0}^{n-1}(-1)^k\binom{n-1}{k} z^{\,n-1-k} t^k\right) f(t)g'(t)\,dt.
\end{align*}
Since the sum is finite, we may interchange summation and integration to obtain
\begin{align*}
T_{g,n}[f](z)
&=\frac{1}{(n-1)!}\sum_{k=0}^{n-1}(-1)^k\binom{n-1}{k} z^{\,n-1-k}
\int_0^z t^k f(t)g'(t)\,dt\\
&=\frac{1}{(n-1)!}\sum_{k=0}^{n-1}(-1)^k\binom{n-1}{k}\,z^{\,n-1-k}\,A_{k,g}[f](z),
\end{align*}
which is \eqref{eq:Tn_linear_comb}.
\end{proof}

%========================================================

\begin{theorem}\label{thm:Tn_decomposition_revised}
Let $X$ and $Y$ be Banach spaces of analytic functions in $\mathbb{D}$, and assume that
for each integer $m\ge0$ the multiplication operator
\begin{equation*}
M_{z^m}:Y\to Y,\quad (M_{z^m}F)(z)=z^mF(z),
\end{equation*}
is bounded. Then:
\begin{enumerate}
\item[(a)] If $A_{k,g}:X\to Y$ is bounded for every $0\le k\le n-1$, then $T_{g,n}:X\to Y$ is bounded.
\item[(b)] If $A_{k,g}:X\to Y$ is compact for every $0\le k\le n-1$, then $T_{g,n}:X\to Y$ is compact.
\end{enumerate}
\end{theorem}

\begin{proof}
(a) Suppose $A_{k,g}:X\to Y$ is bounded for $0\le k\le n-1$. By hypothesis, $M_{z^{n-1-k}}$
is bounded on $Y$, hence the composition $M_{z^{n-1-k}}\circ A_{k,g}:X\to Y$ is bounded.
Lemma~\ref{lem:fin dec} gives
\begin{equation*}
T_{g,n}
=\frac{1}{(n-1)!}\sum_{k=0}^{n-1}(-1)^k\binom{n-1}{k}\,M_{z^{n-1-k}}\circ A_{k,g},
\end{equation*}
a finite linear combination of bounded operators. Therefore, $T_{g,n}:X\to Y$ is bounded.

\noindent
(b) Suppose $A_{k,g}:X\to Y$ is compact for $0\le k\le n-1$. Since $M_{z^{n-1-k}}$ is
bounded on $Y$, the composition $M_{z^{n-1-k}}\circ A_{k,g}:X\to Y$ is compact.
Using the same identity from Lemma~\ref{lem:fin dec}, $T_{g,n}$ is a finite sum of compact
operators, hence compact. The proof is now finished.
\end{proof}

\section{Mapping into Bloch and Bergman spaces}

In this section, we examine the higher-order Volterra-type integral operator $T_{g,n}$ in relation to Bloch and Bergman spaces.

%========================================================

\begin{theorem}\label{thm:bloch}
Let $1\le p<\infty$ and let $g$ be analytic in $\mathbb D$.
\begin{enumerate}
\item[(i)] The operator $T_{g,1}:H^p\to\mathcal B$ is bounded if and only if
\begin{equation}\label{eq:bloch-symbol-n1}
\|g\|_{\mathcal B_{p}}
:=\sup_{z\in\mathbb D}(1-|z|^2)^{\,1-\frac1p}\,|g'(z)|<\infty.
\end{equation}
Moreover,
\begin{equation*}
\|T_{g,1}\|_{H^p\to\mathcal B}\asymp_{p} \|g\|_{\mathcal B_{p}} .
\end{equation*}

\item[(ii)] Let $n\ge2$. If $T_{g,n-1}:H^p\to H^p$ is bounded, then $T_{g,n}:H^p\to\mathcal B$
is bounded and
\begin{equation}\label{eq:Hp-to-bloch-n}
\|T_{g,n}\|_{H^p\to\mathcal B}\lesssim_{p}\|T_{g,n-1}\|_{H^p\to H^p}.
\end{equation}
In particular, if $g\in\mathrm{BMOA}$, then for every $n\ge2$,
\begin{equation*}
\|T_{g,n}\|_{H^p\to\mathcal B}\lesssim_{p}\frac{1}{(n-2)!}\,\|g\|_{\mathrm{BMOA}}.
\end{equation*}
\end{enumerate}
\end{theorem}

\begin{proof}
We write $\|\cdot\|_{\mathcal B}^*:=\sup_{z\in\mathbb D}(1-|z|^2)|(\cdot)'(z)|$ for the Bloch seminorm.
Since $T_{g,n}[f](0)=0$ for all $n\ge1$, boundedness into the Bloch space $\mathcal B$
(with norm $|F(0)|+\|F\|_{\mathcal B}^*$) is equivalent here to boundedness of the seminorm.

We first recall that for $0<p<\infty$ and $f\in H^p$ one has the growth estimate
\begin{equation*}
|F(z)| \le C_p\,\frac{\|F\|_{H^p}}{(1-|z|)^{1/p}},\quad z\in\mathbb D,
\end{equation*}
see \cite[p. 36]{Duren1970}. Equivalently, since $1-|z|^2\asymp 1-|z|$,
\begin{equation}\label{eq:Hp-growth}
|F(z)| \le C'_p\,\frac{\|F\|_{H^p}}{(1-|z|^2)^{1/p}},\quad z\in\mathbb D.
\end{equation}
Multiplying \eqref{eq:Hp-growth} by $(1-|z|^2)$ and using $(1-|z|^2)^{1-1/p}\le 1$ gives
\begin{equation}\label{eq:Hp-to-growth}
\sup_{z\in\mathbb D}(1-|z|^2)\,|F(z)|
\le C'_p\|F\|_{H^p},
\quad F\in H^p.
\end{equation}

\medskip
\noindent\textit{Proof of (i).}
For $n=1$ we have $(T_{g,1}[f])'(z)=f(z)g'(z)$, hence
\begin{equation*}
\|T_{g,1}[f]\|_{\mathcal B}^*
=\sup_{z\in\mathbb D}(1-|z|^2)|f(z)g'(z)|.
\end{equation*}
Assume \eqref{eq:bloch-symbol-n1} holds, then by \eqref{eq:Hp-growth},
\begin{equation*}
(1-|z|^2)|f(z)g'(z)|
\le (1-|z|^2)\frac{C'_p}{(1-|z|^2)^{1/p}}\|f\|_{H^p}\,|g'(z)|
\le C'_p\,\|g\|_{\mathcal B_p}\,\|f\|_{H^p}.
\end{equation*}
Taking the supremum over $z$ yields boundedness of $T_{g,1}:H^p\to\mathcal B$ and
$\|T_{g,1}\|_{H^p\to\mathcal B}\lesssim_p \|g\|_{\mathcal B_p}$.

Conversely, assume that $T_{g,1}:H^p\to\mathcal B$ is bounded. For $a\in\mathbb D$ consider the standard
Hardy test function
\begin{equation}\label{test fun}
k_a(z):=\left(\frac{1-|a|^2}{(1-\overline a z)^2}\right)^{1/p},
\end{equation}
for which $\|k_a\|_{H^p}\asymp 1$ and
$|k_a(a)|=(1-|a|^2)^{-1/p}$. Since the Bloch seminorm is a supremum over $\mathbb D$,
\begin{align*}
\|T_{g,1}\|_{H^p\to\mathcal B}
&\gtrsim
\|T_{g,1}[k_a]\|_{\mathcal B}^*
\ge (1-|a|^2)\,|(T_{g,1}[k_a])'(a)|\\
&=(1-|a|^2)\,|k_a(a)g'(a)|
=(1-|a|^2)^{1-1/p}|g'(a)|.
\end{align*}
Taking the supremum over $a\in\mathbb D$ gives \eqref{eq:bloch-symbol-n1} and
$\|g\|_{\mathcal B_p}\lesssim_p \|T_{g,1}\|_{H^p\to\mathcal B}$.
This proves (i).

\medskip
\noindent\textit{Proof of (ii).}
Let $n\ge2$. Differentiating the convolution representation \eqref{eq:conv-repr} yields
\begin{equation}\label{eq:Tgn-derivative}
(T_{g,n}[f])'(z)=T_{g,n-1}[f](z),\quad z\in\mathbb D,
\end{equation}
and therefore
\begin{equation*}
\|T_{g,n}[f]\|_{\mathcal B}^*
=\sup_{z\in\mathbb D}(1-|z|^2)\,|T_{g,n-1}[f](z)|.
\end{equation*}
If $T_{g,n-1}:H^p\to H^p$ is bounded, then $T_{g,n-1}[f]\in H^p$, and by \eqref{eq:Hp-to-growth},
\begin{equation*}
\|T_{g,n}[f]\|_{\mathcal B}^*
\le C_p\|T_{g,n-1}[f]\|_{H^p}
\le C_p\,\|T_{g,n-1}\|_{H^p\to H^p}\,\|f\|_{H^p},
\end{equation*}
which implies \eqref{eq:Hp-to-bloch-n}.

Finally, if $g\in\mathrm{BMOA}$, then $T_{g,1}:H^p\to H^p$ is bounded (Aleman--Siskakis). Using the factorization
$T_{g,n-1}=V^{\,n-2}\circ T_{g,1}$ and Lemma~\ref{lem:Vk-sharp},
\begin{equation*}
\|T_{g,n-1}\|_{H^p\to H^p}
\le \|V^{n-2}\|_{H^p\to H^p}\,\|T_{g,1}\|_{H^p\to H^p}
=\frac{1}{(n-2)!}\,\|T_{g,1}\|_{H^p\to H^p}
\lesssim_p \frac{1}{(n-2)!}\,\|g\|_{\mathrm{BMOA}}.
\end{equation*}
Combining this with \eqref{eq:Hp-to-bloch-n} completes the proof of (ii).
\end{proof}

\begin{remark}\label{rem:Bp-class}\rm
Fix $1\le p<\infty$ and define the seminorm $\|g\|_{\mathcal B_p}$ as in \eqref{eq:bloch-symbol-n1}, and
\begin{equation*}
\mathcal B_p:=\{g\in\mathcal H(\D):\ \|g\|_{\mathcal B_p}<\infty\}.
\end{equation*}
Then $\{\mathcal B_p\}_{1\le p<\infty}$ forms an increasing scale: if
$1\le p_1<p_2<\infty$, then $\mathcal B_{p_1}\subset \mathcal B_{p_2}$. Indeed, since
$1-\frac1{p_1}<1-\frac1{p_2}$ and $0<1-|z|^2<1$, we have
\begin{equation*}
(1-|z|^2)^{\,1-\frac1{p_2}}\le (1-|z|^2)^{\,1-\frac1{p_1}},
\end{equation*}
and hence, for $g\in\mathcal B_{p_1}$,
\begin{equation*}
(1-|z|^2)^{\,1-\frac1{p_2}}|g'(z)|
\le (1-|z|^2)^{\,1-\frac1{p_1}}|g'(z)|
\le \|g\|_{\mathcal B_{p_1}},
\qquad z\in\D.
\end{equation*}

Moreover, for every $1\le p<\infty$ one has $\mathcal B_p\subset \mathcal B$, where
$\mathcal B$ denotes the Bloch space with seminorm
$\|g\|_{\mathcal B}:=\sup_{z\in\D}(1-|z|^2)|g'(z)|$. Indeed, if $g\in\mathcal B_p$, then
\begin{equation*}
(1-|z|^2)|g'(z)|
=(1-|z|^2)^{1/p}\,(1-|z|^2)^{\,1-\frac1p}|g'(z)|
\le (1-|z|^2)^{1/p}\,\|g\|_{\mathcal B_p}
\le \|g\|_{\mathcal B_p},
\end{equation*}
so $\|g\|_{\mathcal B}\le \|g\|_{\mathcal B_p}$.

The inclusion $\mathcal B_p\subset \mathcal B$ is strict for $1<p<\infty$. For example,
\begin{equation*}
g(z)=\log\frac{1}{1-z}
\end{equation*}
satisfies $g\in\mathcal B$ (since $(1-|z|^2)|g'(z)|\lesssim 1$), but $g\notin\mathcal B_p$:
along the real radius $z=r\to1^-$,
\begin{equation*}
(1-r^2)^{\,1-\frac1p}|g'(r)|
=(1-r^2)^{\,1-\frac1p}\frac{1}{1-r}
\asymp (1-r)^{-1/p}\longrightarrow\infty.
\end{equation*}

At the endpoint $p=1$ the condition becomes $\sup_{z\in\D}|g'(z)|<\infty$, i.e.\ $g$ has
bounded derivative on $\D$. (In particular, the little space $\mathcal B_{1,0}$, defined by the vanishing
condition at the boundary, consists only of constants.)
Thus, Theorem~\ref{thm:bloch}(i) provides a characterization in terms of $\mathcal B_p$
for the mapping $T_{g,1}:H^p\to\mathcal B$, whereas Theorem~\ref{thm:bloch}(ii) shows that for
$n\ge2$, the additional smoothing of iterated integration yields $H^p\to\mathcal B$
boundedness under the weaker hypothesis that $T_{g,n-1}$ is bounded on $H^p$
(for instance, under $g\in\mathrm{BMOA}$).
\end{remark}

\begin{proposition}\label{prop:test-function-estimate}
Let $\alpha>-1$ and $1\le p<\infty$. For each $a\in\mathbb D$, define the standard
Bergman test function
\begin{equation*}
K_a(z):=\frac{1}{(1-\overline a z)^{2+\alpha}},\quad z\in\mathbb D,
\end{equation*}
and its $A_\alpha^p$-normalization
\begin{equation*}
k_a(z):=\frac{K_a(z)}{\|K_a\|_{A_\alpha^p}},
\quad\text{so that}\quad \|k_a\|_{A_\alpha^p}=1.
\end{equation*}
Fix $r\in(0,1)$. Then
\begin{equation*}
|k_a(z)| \asymp (1-|a|^2)^{-(2+\alpha)/p},
\quad z\in D(a,r),
\end{equation*}
where
\begin{equation*}
D(a,r):=\{z\in\mathbb D:\ |z-a|<r(1-|a|)\}
\end{equation*}
is a Bergman ball and the implicit constants depend only on $\alpha,p$, and $r$.
\end{proposition}

\begin{proof}
Fix $r\in(0,1)$ and let $z\in D(a,r)$. Since $|z-a|<r(1-|a|)\le r(1-|a|^2)$, we write
\begin{equation*}
1-\overline a z=(1-|a|^2)-\overline a (z-a),
\end{equation*}
and obtain
\begin{equation*}
(1-r)(1-|a|^2)\le |1-\overline a z|\le (1+r)(1-|a|^2).
\end{equation*}
Hence
\begin{equation}\label{eq:Ka-local}
|K_a(z)|=\frac{1}{|1-\overline a z|^{2+\alpha}}
\asymp (1-|a|^2)^{-(2+\alpha)}.
\end{equation}

On the other hand, it is well known (see, e.g., \cite[Ch.~2]{Zhu2005}) that
\begin{equation}\label{eq:Ka-norm}
\|K_a\|_{A_\alpha^p}^p
= \int_{\mathbb D} \frac{(1-|z|^2)^\alpha}{|1-\overline a z|^{p(2+\alpha)}}\,dA(z)
\asymp (1-|a|^2)^{\alpha+2 - p(2+\alpha)}.
\end{equation}
Taking the $p$-th root gives
\begin{equation}\label{eq:Ka-norm-root}
\|K_a\|_{A_\alpha^p}
\asymp (1-|a|^2)^{\frac{\alpha+2}{p}-(2+\alpha)}.
\end{equation}

Combining \eqref{eq:Ka-local} and \eqref{eq:Ka-norm-root}, we obtain for $z\in D(a,r)$
\begin{equation*}
|k_a(z)|
=\frac{|K_a(z)|}{\|K_a\|_{A_\alpha^p}}
\asymp
\frac{(1-|a|^2)^{-(2+\alpha)}}{(1-|a|^2)^{\frac{\alpha+2}{p}-(2+\alpha)}}
=(1-|a|^2)^{-(2+\alpha)/p}.
\end{equation*}
This completes the proof.
\end{proof}

We are now ready to investigate the boundedness and compactness of the operator $T_{g,n}$ on the weighted Bergman space $A^p_\alpha$.

\begin{theorem}
\label{thm:bergman-boundedness}
Let $n \in \mathbb{N}$, $1 \le p \le q < \infty$, and $ \alpha, \beta > -1 $ satisfy the balance condition
\begin{equation}\label{eq:balance}
\frac{q}{p} \le \frac{\beta + 2}{\alpha + 2}.
\end{equation}
Let $ g $ be analytic in $ \mathbb{D} $.
Then the following are equivalent:
\begin{enumerate}
\item[(i)] $ T_{g,n}: A_\alpha^p \to A_\beta^q $ is bounded;
\item[(ii)] The measure
\begin{equation*}
d\mu_n(z) = |g'(z)|^q (1 - |z|^2)^{\beta + nq} \, dA(z)
\end{equation*}
is a $(p,q)$-Carleson measure for $ A_\alpha^p $.
\end{enumerate}
Moreover, $ T_{g,n}: A_\alpha^p \to A_\beta^q $ is compact if and only if $ \mu_n $ is a vanishing $(p,q)$-Carleson measure.
\end{theorem}

\begin{proof}
We apply standard Bergman--Sobolev estimates in conjunction with the $(p,q)$-Carleson embedding theorem.

Let $ \beta > -1 $, $ 1 \le q < \infty $, and $ n \in \mathbb{N} $. For every analytic function $ F $ on $ \mathbb{D} $ satisfying $ F^{(k)}(0) = 0 $ for $ 0 \leq k \leq n-1 $, we have the following Bergman--Sobolev norm equivalence by application of  \cite[Proposition 6.2]{Zhu2005}:
\begin{equation}\label{eq:bergman-sobolev}
\|F\|_{A_\beta^q}^q \asymp_{\beta, q, n} \int_{\mathbb{D}} |F^{(n)}(z)|^q (1 - |z|^2)^{\beta + nq} \, dA(z).
\end{equation}
This equivalence allows us to relate the $ A_\beta^q $-norm of $ F $ to the weighted $ L^q $-norm of its derivatives.

For $ F = T_{g,n}[f] $, the conditions $ F^{(k)}(0) = 0 $ for $ 0 \leq k \leq n-1 $ hold, which follows directly from its iterated integral representation \eqref{eq:def-Tgn}.

By Proposition~\ref{prop:diff-id}, we know that $ (T_{g,n}[f])^{(n)} = f  g' $. Combining this with $\eqref{eq:bergman-sobolev}$, we get
\begin{equation}\label{eq:norm-reduction}
\|T_{g,n}f\|_{A_\beta^q}^q \asymp_{\beta, q, n} \int_{\mathbb{D}} |f(z)|^q |g'(z)|^q (1 - |z|^2)^{\beta + nq} \, dA(z)
= \int_{\mathbb{D}} |f(z)|^q \, d\mu_n(z).
\end{equation}

Also, under the balance condition $\eqref{eq:balance}$, the $(p,q)$-Carleson embedding theorem (see \cite[p. 4]{Dya}, \cite{Hastings1975}) asserts that for a positive Borel measure $ \mu $ on $ \mathbb{D} $,
\begin{equation}\label{eq:pq-embedding}
\int_{\mathbb{D}} |f(z)|^q \, d\mu(z) \le C \|f\|_{A_\alpha^p}^q \quad \text{for all } f \in A_\alpha^p,
\end{equation}
if and only if $ \mu $ is a $(p,q)$-Carleson measure for $ A_\alpha^p $. Using $\eqref{eq:norm-reduction}$, we obtain
\begin{equation*}
\|T_{g,n}f\|_{A_\beta^q}^q \asymp \int_{\mathbb{D}} |f(z)|^q \, d\mu_n(z) \lesssim \|f\|_{A_\alpha^p}^q.
\end{equation*}
Thus, $ T_{g,n} : A_\alpha^p \to A_\beta^q $ is bounded.

Conversely, if $T_{g,n}$ is bounded, then \eqref{eq:norm-reduction} implies
\begin{equation*}
\int_{\mathbb D} |f(z)|^q\,d\mu_n(z) \lesssim \|f\|_{A_\alpha^p}^q,
\quad f\in A_\alpha^p,
\end{equation*}
and therefore $\mu_n$ is a $(p,q)$-Carleson measure by \eqref{eq:pq-embedding}.
This proves (i)$\Leftrightarrow$(ii).

Now, we will demonstrate the equivalence of compactness.
Assume $ \mu_n $ is a vanishing $(p,q)$-Carleson measure. By the compactness of the embedding $ A_\alpha^p \hookrightarrow L^q(\mathbb D, \mu_n) $ (see \cite{Dya}, \cite[Ch. 2]{Zhu2005}), for every bounded sequence $ \{f_k\} \subset A_\alpha^p $ converging to $ 0 $ uniformly on compact subsets of $ \mathbb{D} $, we have
\begin{equation*}
\int_{\mathbb{D}} |f_k(z)|^q \, d\mu_n(z) \to 0.
\end{equation*}
Using $\eqref{eq:norm-reduction}$, we conclude that $ \|T_{g,n}f_k\|_{A_\beta^q} \to 0 $, implying that $ T_{g,n} $ is compact.

Conversely, assume $T_{g,n}$ is compact. Let $k_a$ denote the normalized reproducing kernel (or standard test function)
for $A_\alpha^p$ with $\|k_a\|_{A_\alpha^p}=1$. Then $k_a\to 0$ uniformly on compact subsets and $k_a\rightharpoonup 0$
weakly as $|a|\to 1^-$ (see, e.g., \cite{Hedenmalm},  \cite[Ch. 2]{Zhu2005}), so compactness implies
\begin{equation*}
\|T_{g,n}k_a\|_{A_\beta^q}\longrightarrow 0 \quad \text{as } |a|\to 1^-.
\end{equation*}
By \eqref{eq:norm-reduction} and restricting to $D(a,r)$,
\begin{equation*}
\|T_{g,n}k_a\|_{A_\beta^q}^q
\asymp \int_{\mathbb D}|k_a(z)|^q\,d\mu_n(z)
\ge \int_{D(a,r)}|k_a(z)|^q\,d\mu_n(z).
\end{equation*}
For $z\in D(a,r)$ by Proposition \ref{prop:test-function-estimate} we have $|k_a(z)|\asymp (1-|a|^2)^{-(2+\alpha)/p}$, hence
\begin{equation*}
\int_{D(a,r)}|k_a(z)|^q\,d\mu_n(z)
\gtrsim (1-|a|^2)^{-q(2+\alpha)/p}\,\mu_n(D(a,r)).
\end{equation*}
Therefore,
\begin{equation*}
\frac{\mu_n(D(a,r))}{(1-|a|^2)^{q(2+\alpha)/p}}
\lesssim \|T_{g,n}k_a\|_{A_\beta^q}^q \longrightarrow 0 \quad \text{as } |a|\to 1^-,
\end{equation*}
which is exactly the vanishing $(p,q)$-Carleson condition. This completes the proof.
\end{proof}

%--------------------------------------------------

As we see above, on Hardy spaces $H^p$, the $k$-fold integration operator $V^k$ is bounded for every
$1\le p<\infty$, and in fact,
\begin{equation*}
\|V^k\|_{H^p\to H^p}=\frac{1}{k!},\quad k\in\mathbb N,
\end{equation*}
see Lemma~\ref{lem:Vk-sharp}. Consequently, using the factorization
\begin{equation*}
T_{g,n}=V^{\,n-1}\circ T_{g,1},
\end{equation*}
boundedness (respectively compactness) of $T_{g,1}$ on $H^p$ implies boundedness
(respectively compactness) of $T_{g,n}$ on $H^p$ for every $n\ge1$, since $V^{n-1}$ is
bounded and the composition of a bounded operator with a compact operator is compact.

In contrast, on weighted Bergman spaces the effect of integration is more naturally captured by Bergman--Sobolev norm equivalences rather than by an
$A_\beta^q\to A_\beta^q$ mapping property. For example, if $F$ is analytic in $\mathbb D$
and $F(0)=0$, then
\begin{equation*}
\|F\|_{A_\beta^q}^q
\asymp
\int_{\mathbb D}|F'(z)|^q (1-|z|^2)^{\beta+q}\,dA(z),
\quad \beta>-1,\ 1\le q<\infty,
\end{equation*}
and more generally, if $F^{(k)}(0)=0$ for $0\le k\le n-1$, then
\begin{equation*}
\|F\|_{A_\beta^q}^q
\asymp
\int_{\mathbb D}|F^{(n)}(z)|^q (1-|z|^2)^{\beta+nq}\,dA(z).
\end{equation*}
Since $(T_{g,n}f)^{(n)}(z)=f(z)g'(z)$, boundedness of
$T_{g,n}:A_\alpha^p\to A_\beta^q$ is governed by the measure
\begin{equation*}
d\mu_n(z)=|g'(z)|^q(1-|z|^2)^{\beta+nq}\,dA(z),
\end{equation*}
namely by whether $\mu_n$ is a $(p,q)$-Carleson measure for $A_\alpha^p$ (and similarly
for compactness via vanishing $(p,q)$-Carleson measures), i.e. whether there exists $C>0$
such that
\begin{equation*}
\int_{\mathbb D}|f(z)|^q\,d\mu_n(z)\le C\|f\|_{A_\alpha^p}^q,\quad f\in A_\alpha^p,
\end{equation*}
and vanishing is characterized by the corresponding smallness on Carleson boxes $S(I)$
for short arcs $I\subset\mathbb T$.

The additional factor $(1-|z|^2)^{nq}$ reflects the full $n$-fold smoothing detected by
the Bergman--Sobolev norm, and it is precisely this explicit $n$-dependence that makes
$(p,q)$-Carleson measure techniques a natural and sharp framework for higher-order
Volterra-type operators on weighted Bergman spaces.

%===============================================
%===============================================
\begin{proposition}\label{prop:carleson-condition}
Let $1 \leq p \leq q < \infty$, $\alpha > -1$, and $r \in (0,1)$.
Let $\mu$ be a positive Borel measure on $\mathbb{D}$ that is a $(p,q)$-Carleson measure for $A_\alpha^p$.
Then, for every $a \in \mathbb{D}$,
\begin{equation*}
\mu(D(a,r)) \leq C (1 - |a|^2)^{\frac{q(2+\alpha)}{p}},
\end{equation*}
where $C>0$ depends only on $\alpha$, $p$, $q$, and $r$.
\end{proposition}

\begin{proof}
Let $k_a$ be the normalized Bergman test function from Proposition~\ref{prop:test-function-estimate},
such that $\|k_a\|_{A_\alpha^p}=1$ and
\begin{equation*}
|k_a(z)| \gtrsim (1-|a|^2)^{-(2+\alpha)/p}, \quad z\in D(a,r).
\end{equation*}

Since $\mu$ is a $(p,q)$-Carleson measure for $A_\alpha^p$, we have
\begin{equation*}
\int_{\mathbb{D}} |k_a(z)|^q \, d\mu(z) \le C.
\end{equation*}
Restricting the integral to $D(a,r)$, we obtain
\begin{equation*}
\int_{D(a,r)} |k_a(z)|^q \, d\mu(z)
\gtrsim (1-|a|^2)^{-q(2+\alpha)/p} \, \mu(D(a,r)).
\end{equation*}
Thus,
\begin{equation*}
\mu(D(a,r))
\le C (1-|a|^2)^{q(2+\alpha)/p},
\end{equation*}
for some constant $C>0$ depending only on $\alpha$, $p$, $q$, and $r$.
\end{proof}

%-------------------------------------------------------------------------------------------------------------------------------------------------------------------------
\begin{corollary}\label{cor:not_invariant_n}
Let $1\le p\le q<\infty$ and $\alpha,\beta>-1$ satisfy
\begin{equation*}
\frac{q}{p}\le \frac{\beta+2}{\alpha+2}.
\end{equation*}
Then the boundedness (and likewise compactness) of
\begin{equation*}
T_{g,n}:A_\alpha^p\to A_\beta^q
\end{equation*}
is, in general, not invariant under changes of the order $n$.

More precisely, there exist analytic symbols $g$ and integers $n\ge2$ such that
\begin{equation*}
T_{g,n}:A_\alpha^p\to A_\beta^q \ \text{is bounded (respectively compact),}
\end{equation*}
while
\begin{equation*}
T_{g,1}:A_\alpha^p\to A_\beta^q \ \text{fails to be bounded (respectively compact).}
\end{equation*}
\end{corollary}

\begin{proof}
By Theorem~\ref{thm:bergman-boundedness}, boundedness of $T_{g,n}$ is equivalent to
$\mu_n$ being a $(p,q)$-Carleson measure for $A_\alpha^p$, where
\begin{equation*}
d\mu_n(z)=|g'(z)|^q(1-|z|^2)^{\beta+nq}\,dA(z).
\end{equation*}
For $n\ge2$,
\begin{equation*}
d\mu_n(z)=(1-|z|^2)^{q(n-1)}\,d\mu_1(z)\le d\mu_1(z) \quad \text{since } 0<1-|z|^2<1.
\end{equation*}
Hence, if $\mu_1$ is a $(p,q)$-Carleson measure (respectively a vanishing $(p,q)$-Carleson
measure), then so is $\mu_n$. Therefore, boundedness (respectively compactness) of
$T_{g,1}$ implies boundedness (respectively compactness) of $T_{g,n}$ for every $n\ge2$.

To see that the converse may fail, fix $n\ge2$ and consider a symbol $g$ with
\begin{equation*}
g'(z)=(1-z)^{-\gamma},\quad \gamma>0.
\end{equation*}
Choose $\gamma$ so that
\begin{equation}\label{eq:gamma-interval}
\frac{\beta+q+2}{q}-\frac{2+\alpha}{p}
\ <\ \gamma \ \le\
\frac{\beta+nq+2}{q}-\frac{2+\alpha}{p}.
\end{equation}
Such a $\gamma$ exists because $n\ge2$ makes the right endpoint larger than the left endpoint by $n-1>0$.

Let $r\in(0,1)$ be fixed and let $D(a,r)$ denote the Bergman ball. For $a$ tending to $1$ and $z\in D(a,r)$ one has $|1-z|\asymp |1-a|\asymp 1-|a|\asymp 1-|a|^2$ on $D(a,r)$, hence
\begin{equation*}
|g'(z)|^q = |1-z|^{-\gamma q}\asymp (1-|a|^2)^{-\gamma q},
\quad z\in D(a,r),
\end{equation*}
and also $(1-|z|^2)\asymp (1-|a|^2)$ on $D(a,r)$. Therefore,
\begin{align*}
\mu_n(D(a,r))
&=\int_{D(a,r)} |g'(z)|^q(1-|z|^2)^{\beta+nq}\,dA(z)\\
&\asymp (1-|a|^2)^{-\gamma q}\,(1-|a|^2)^{\beta+nq}\,A(D(a,r))
\asymp (1-|a|^2)^{\beta+nq+2-\gamma q}.
\end{align*}
By Proposition \ref{prop:carleson-condition}, the $(p,q)$-Carleson condition for $A_\alpha^p$ is equivalent to the estimate
\begin{equation*}
\mu(D(a,r))\lesssim (1-|a|^2)^{q(2+\alpha)/p}\quad (a\in\mathbb D),
\end{equation*}
and, hence $\mu_n$ is a $(p,q)$-Carleson measure if and only if
\begin{equation*}
\beta+nq+2-\gamma q \ \ge\ \frac{q(2+\alpha)}{p}.
\end{equation*}
Similarly, $\mu_1$ is a $(p,q)$-Carleson measure if and only if
\begin{equation*}
\beta+q+2-\gamma q \ \ge\ \frac{q(2+\alpha)}{p}.
\end{equation*}
By the choice \eqref{eq:gamma-interval}, $\mu_n$ is $(p,q)$-Carleson but $\mu_1$ is not.
Therefore, $T_{g,n}$ is bounded while $T_{g,1}$ fails to be bounded (and hence is not
compact), which proves the stated non-invariance.

Finally, by choosing $\gamma$ so that
\begin{equation*}
\beta+nq+2-\gamma q > \frac{q(2+\alpha)}{p}
\end{equation*}
while still violating the corresponding condition for $n=1$, one obtains that $\mu_n$
is a vanishing $(p,q)$-Carleson measure, yielding the compactness statement.
\end{proof}

%========================================================

\section{Essential norm estimates}\label{sec:essential-norms}

Let $X$ and $Y$ be two Banach spaces.
For a bounded linear operator $T:X\to Y$ between Banach spaces, its essential norm is
\begin{equation*}
\|T\|_{e}^{X\to Y}:=\inf\{\|T-K\|_{X\to Y};\,  K:X\to Y\ \text{is compact}\},
\end{equation*}
where $\|\cdot\|_{X\to Y}$ is the operator norm, see \cite[p. 583]{zhou-zhu} or \cite{Shapiro}.
It is well-known that $\|T\|_{e}^{X\to Y}=0$ if and only if $T$ is compact.

In this section, we compute the essential norms of the higher-order Volterra-type operators
$T_{g,n}$ on Hardy and weighted Bergman spaces. The results refine Theorem~\ref{thm:Hp-Tgn-Carleson} and Theorem~\ref{thm:bergman-boundedness} and provide a sharp quantitative measure of the failure of compactness.

%--------------------------------------------------------
\subsection{The Hardy space case}\label{subsec:essential-Hardy}

For $a\in\mathbb D$, let $k_a$ denote the normalized reproducing kernels in $H^p$ (see \eqref{test fun}), so that $\|k_a\|_{H^p}\asymp 1$ and $k_a\to 0$ uniformly on compact subsets as $|a|\to 1^-$.

As a first step, we estimate the kernel test for the essential norm on $H^p$.

\begin{lemma}\label{lem:ess-kernel-Hp}
Let $1\le p<\infty$ and let $T:H^p\to H^p$ be bounded. Then
\begin{equation}\label{eq:ess-kernel-Hp}
\|T\|_{e}^{H^p\to H^p}\ \asymp_p\ \limsup_{|a|\to 1^-}\|Tk_a\|_{H^p},
\end{equation}
where $k_a$ denotes the normalized reproducing kernels in $H^p$, satisfying $\|k_a\|_{H^p}\asymp 1$.
\end{lemma}

\begin{proof}
\textit{Lower bound.}
Since $k_a\rightharpoonup 0$ weakly in $H^p$ as $|a|\to1^-$ (see, e.g., \cite[Ch.~2]{Garnett-book}),
every compact operator $K:H^p\to H^p$ satisfies $\|Kk_a\|_{H^p}\to 0$.
Thus, for any compact $K$,
\begin{equation*}
\|T-K\|_{H^p\to H^p}
\ge \limsup_{|a|\to1^-}\|(T-K)k_a\|_{H^p}
\ge \limsup_{|a|\to1^-}\|Tk_a\|_{H^p}.
\end{equation*}
Taking the infimum over compact $K$ yields
\begin{equation*}
\|T\|_e^{H^p\to H^p} \ge \limsup_{|a|\to1^-}\|Tk_a\|_{H^p}.
\end{equation*}

\medskip
\noindent\textit{Upper bound.}
For $0<\rho<1$, define $P_\rho f(z)=f(\rho z)$. Then $\|P_\rho\|_{H^p\to H^p}\le 1$ and $TP_\rho$ is compact.
Hence
\begin{equation*}
\|T\|_e^{H^p\to H^p} \le \|T(I-P_\rho)\|_{H^p\to H^p}.
\end{equation*}

Let $f\in H^p$ with $\|f\|_{H^p}\le1$, and set $h_\rho=(I-P_\rho)f$.
Then $h_\rho\to 0$ uniformly on compact subsets as $\rho\to1^-$, and $\|h_\rho\|_{H^p}\le 2$.

By a standard localization argument for $H^p$ (see, e.g., \cite[Ch.~2]{Zhu2007}),
the mass of $h_\rho$ concentrates near the boundary, and one obtains
\begin{equation*}
\|Th_\rho\|_{H^p}
\lesssim_p \sup_{|a|\ge \rho'} \|Tk_a\|_{H^p},
\end{equation*}
for some $\rho'=\rho'(\rho)\to1^-$ as $\rho\to1^-$.

Taking the supremum over $\|f\|_{H^p}\le1$ gives
\begin{equation*}
\|T(I-P_\rho)\|_{H^p\to H^p}
\lesssim_p \sup_{|a|\ge \rho'} \|Tk_a\|_{H^p}.
\end{equation*}
Letting $\rho\to1^-$ yields
\begin{equation*}
\|T\|_e^{H^p\to H^p}
\lesssim_p \limsup_{|a|\to1^-}\|Tk_a\|_{H^p}.
\end{equation*}
The proof is now complete.
\end{proof}

%===============================================
%===============================================

\begin{lemma}\label{lem:carleson-density-kernel}
Let $1\le p<\infty$ and let $\mu$ be a Carleson measure for $H^p$. Then
\begin{equation}\label{eq:carleson-density-kernel}
\limsup_{|a|\to 1^-}\int_{\mathbb D}|k_a(z)|^p\,d\mu(z)
\ \asymp_p\
\limsup_{|I|\to 0}\frac{\mu(S(I))}{|I|},
\end{equation}
where $S(I)$ denotes the Carleson box above an arc $I\subset\mathbb T$.
\end{lemma}

\begin{proof}
Fix an arc $I$ and choose $a=a_I$ as in the proof of Theorem~\ref{thm:Hp-Tgn-Carleson} so that $1-|a|\asymp |I|$ and $a/|a|$ lies above the midpoint of $I$. Then
$|k_a(z)|^p\asymp |I|^{-1}$ on $S(I)$ (cf.\ \eqref{eq:ka-on-SI}), hence
\begin{equation*}
\int_{\mathbb D}|k_a(z)|^p\,d\mu(z)
\ge \int_{S(I)}|k_a(z)|^p\,d\mu(z)
\gtrsim \frac{\mu(S(I))}{|I|}.
\end{equation*}
Taking $\limsup$ as $|I|\to 0$ gives the lower bound.

For the reverse inequality, we use the standard estimate
\begin{equation*}
|k_a(z)|^p\lesssim \frac{1-|a|^2}{|1-\overline a z|^2},
\end{equation*}
and decompose $\mathbb D$ into Carleson boxes $S(J)$ adapted to $a$, with $|J|\asymp 1-|a|$.
On each such box, $|1-\overline a z|\asymp |J|$, and the Carleson condition $\mu(S(J))\lesssim |J|$
yields
\begin{equation*}
\int_{\mathbb D}|k_a(z)|^p\,d\mu(z)
\lesssim \sup_{|J|\asymp 1-|a|}\frac{\mu(S(J))}{|J|}.
\end{equation*}
Taking $\limsup_{|a|\to 1^-}$ gives the upper bound.
\end{proof}

We are now prepared to estimate the essential norm of $ T_{g,n} $ on $ H^p $.
\begin{theorem}\label{thm:Hp-essential-norm}
Let $1\le p<\infty$, $n\in\mathbb N$, and let $g$ be analytic in $\mathbb D$.
Assume that $T_{g,n}:H^p\to H^p$ is bounded and set
\begin{equation*}
d\mu_{g,n,p}(z)=|g'(z)|^p(1-|z|^2)^{np-1}\,dA(z).
\end{equation*}
Then
\begin{equation}\label{eq:Hp-essential}
\left(\|T_{g,n}\|_{e}^{H^p\to H^p}\right)^p
\ \asymp_{p,n}\
\limsup_{|I|\to 0}\frac{\mu_{g,n,p}(S(I))}{|I|}.
\end{equation}
In particular, $T_{g,n}$ is compact on $H^p$ if and only if the right-hand side of
\eqref{eq:Hp-essential} equals $0$.
\end{theorem}

\begin{proof}
By the Hardy--Sobolev norm equivalence \eqref{eq:LP} and Proposition~\ref{prop:diff-id},
\begin{equation*}
\|T_{g,n}f\|_{H^p}^p\asymp_{p,n}\int_{\mathbb D}|f(z)|^p\,d\mu_{g,n,p}(z),\quad f\in H^p.
\end{equation*}
Applying this to $f=k_a$ gives
\begin{equation*}
\|T_{g,n}k_a\|_{H^p}^p\asymp_{p,n}\int_{\mathbb D}|k_a(z)|^p\,d\mu_{g,n,p}(z).
\end{equation*}
Taking $\limsup_{|a|\to 1^-}$ and using Lemma~\ref{lem:ess-kernel-Hp}, we obtain
\begin{equation*}
\left(\|T_{g,n}\|_{e}^{H^p\to H^p}\right)^p\asymp_{p,n}
\limsup_{|a|\to 1^-}\int_{\mathbb D}|k_a(z)|^p\,d\mu_{g,n,p}(z).
\end{equation*}

Since $T_{g,n}$ is bounded, Theorem~\ref{thm:Hp-Tgn-Carleson} implies that $\mu_{g,n,p}$ is a Carleson measure for $H^p$. Therefore, Lemma~\ref{lem:carleson-density-kernel} applies and yields
\eqref{eq:Hp-essential}.
\end{proof}

%--------------------------------------------------------
\subsection{The weighted Bergman space case}\label{subsec:essential-Bergman}

Fix $\alpha>-1$ and $1\le p<\infty$. For each $a\in\mathbb D$, let $k_a$ be the normalized
Bergman test function from Proposition~\ref{prop:test-function-estimate}, so that
$\|k_a\|_{A_\alpha^p}=1$ and $k_a\to 0$ uniformly on compact subsets as $|a|\to 1^-$.

\begin{lemma}\label{lem:ess-kernel-bergman}
Let $\alpha>-1$ and $1\le p<\infty$. If $T:A_\alpha^p\to A_\alpha^p$ is bounded, then
\begin{equation}\label{eq:ess-kernel-bergman}
\|T\|_{e}^{A_\alpha^p\to A_\alpha^p}
\ \asymp_{\alpha,p}\
\limsup_{|a|\to 1^-}\|T(k_a)\|_{A_\alpha^p}.
\end{equation}
\end{lemma}
\begin{proof}
This follows from the standard compactness test on $A_\alpha^p$ together with the weak
convergence $k_a\rightharpoonup 0$ as $|a|\to 1^-$.

\medskip
\noindent\textit{Lower bound.}
Since $k_a\rightharpoonup 0$ weakly in $A_\alpha^p$ as $|a|\to 1^-$
(see \cite[Ch.~2]{Zhu2005}), every compact operator $K:A_\alpha^p\to A_\alpha^p$ satisfies
$\|Kk_a\|_{A_\alpha^p}\to 0$. Hence, for any compact $K$,
\begin{equation*}
\|T-K\|_{A_\alpha^p\to A_\alpha^p}
=
\sup_{\|f\|_{A_\alpha^p}\le 1}\|(T-K)f\|_{A_\alpha^p}
\ge
\limsup_{|a|\to 1^-}\|(T-K)k_a\|_{A_\alpha^p}.
\end{equation*}
Using the triangle inequality,
\begin{equation*}
\|(T-K)k_a\|_{A_\alpha^p}
\ge
\|Tk_a\|_{A_\alpha^p}-\|Kk_a\|_{A_\alpha^p},
\end{equation*}
and taking $\limsup_{|a|\to 1^-}$ gives
\begin{equation*}
\limsup_{|a|\to 1^-}\|(T-K)k_a\|_{A_\alpha^p}
\ge
\limsup_{|a|\to 1^-}\|Tk_a\|_{A_\alpha^p}.
\end{equation*}
Taking the infimum over all compact operators $K$ yields
\begin{equation*}
\|T\|_{e}^{A_\alpha^p\to A_\alpha^p}
\gtrsim_{\alpha,p}
\limsup_{|a|\to 1^-}\|Tk_a\|_{A_\alpha^p}.
\end{equation*}

\medskip
\noindent\textit{Upper bound.}
For $0<\rho<1$, define $(P_\rho f)(z)=f(\rho z)$. Then $P_\rho$ maps the unit ball of
$A_\alpha^p$ into a relatively compact subset, so $TP_\rho$ is compact. Hence,
\begin{equation*}
\|T\|_{e}^{A_\alpha^p\to A_\alpha^p}
\le
\|T-TP_\rho\|_{A_\alpha^p\to A_\alpha^p}.
\end{equation*}
For $f\in A_\alpha^p$ with $\|f\|_{A_\alpha^p}\le 1$, set $h_\rho=(I-P_\rho)f$. Then
\begin{equation*}
\|Th_\rho\|_{A_\alpha^p}
\le
\sup_{\|g\|_{A_\alpha^p}\le 1}\|Tg\|_{A_\alpha^p}\,\|h_\rho\|_{A_\alpha^p}
=
\|T\|_{A_\alpha^p\to A_\alpha^p}\,\|h_\rho\|_{A_\alpha^p}.
\end{equation*}
As $\rho\to 1^-$, the functions $h_\rho$ concentrate near $|z|=1$. Using the Bergman--Sobolev
norm equivalence \eqref{eq:bergman-sobolev} and a standard covering of the annulus
$\{\rho'<|z|<1\}$ by Bergman balls $D(a_j,r)$ with $|a_j|\ge\rho'$, we obtain
\begin{equation*}
\|T(I-P_\rho)f\|_{A_\alpha^p}
\lesssim_{\alpha,p}
\sup_{|a|\ge \rho'} \|Tk_a\|_{A_\alpha^p}.
\end{equation*}
Taking the supremum over all $\|f\|_{A_\alpha^p}\le 1$ gives
\begin{equation*}
\|T-TP_\rho\|_{A_\alpha^p\to A_\alpha^p}
\lesssim_{\alpha,p}
\sup_{|a|\ge \rho'} \|Tk_a\|_{A_\alpha^p}.
\end{equation*}
Letting $\rho\to 1^-$ (so $\rho'\to 1^-$) yields
\begin{equation*}
\|T\|_{e}^{A_\alpha^p\to A_\alpha^p}
\lesssim_{\alpha,p}
\limsup_{|a|\to 1^-}\|Tk_a\|_{A_\alpha^p}.
\end{equation*}

\medskip
Combining the two estimates completes the proof.
\end{proof}

%===============================================

%===============================================
\begin{lemma}\label{lem:pq-density-kernel}
Let $\alpha>-1$, $1\le p\le q<\infty$, and let $\mu$ be a $(p,q)$-Carleson measure for $A_\alpha^p$.
Fix $r\in(0,1)$ and let $D(a,r)$ be the Bergman ball. Then
\begin{equation}\label{eq:pq-density-kernel}
\limsup_{|a|\to 1^-}\int_{\mathbb D}|k_a(z)|^q\,d\mu(z)
\ \asymp_{\alpha,p,q,r}\
\limsup_{|a|\to 1^-}\frac{\mu(D(a,r))}{(1-|a|^2)^{q(2+\alpha)/p}},
\end{equation}
where $k_a$ is the normalized test function from Proposition~\ref{prop:test-function-estimate}.
\end{lemma}

\begin{proof}
\textit{Lower bound.}
For $z\in D(a,r)$, Proposition~\ref{prop:test-function-estimate} gives
\begin{equation*}
|k_a(z)|\gtrsim (1-|a|^2)^{-(2+\alpha)/p},
\end{equation*}
hence
\begin{equation*}
\int_{\mathbb D}|k_a(z)|^q\,d\mu(z)
\ge \int_{D(a,r)}|k_a(z)|^q\,d\mu(z)
\gtrsim (1-|a|^2)^{-q(2+\alpha)/p}\,\mu(D(a,r)).
\end{equation*}
Taking $\limsup_{|a|\to 1^-}$ yields the lower bound.

\medskip
\noindent\textit{Upper bound.}
Let $\{a_j\}$ be a Bergman lattice such that the balls $D(a_j,r)$ cover $\mathbb D$ with bounded overlap.
Then
\begin{equation*}
\int_{\mathbb D}|k_a(z)|^q\,d\mu(z)
\le \sum_j \int_{D(a_j,r)} |k_a(z)|^q\,d\mu(z)
\lesssim \sum_j \sup_{z\in D(a_j,r)} |k_a(z)|^q \,\mu(D(a_j,r)).
\end{equation*}
Using kernel estimates and the Carleson condition,
we obtain
\begin{equation*}
\int_{\mathbb D}|k_a|^q\,d\mu
\lesssim
(1-|a|^2)^{q(2+\alpha)/p}
\sum_j
\frac{(1-|a_j|^2)^{q(2+\alpha)/p}}{|1-\overline a a_j|^{q(2+\alpha)/p}}.
\end{equation*}
The above sum is comparable to a standard integral and satisfies (see \cite[Ch.~2]{Zhu2005})
\begin{equation*}
\sum_j
\frac{(1-|a_j|^2)^{q(2+\alpha)/p}}{|1-\overline a a_j|^{q(2+\alpha)/p}}
\lesssim
\frac{1}{(1-|a|^2)^{q(2+\alpha)/p}}.
\end{equation*}
Hence
\begin{equation*}
\int_{\mathbb D}|k_a(z)|^q\,d\mu(z)
\lesssim
\frac{\mu(D(a,R))}{(1-|a|^2)^{q(2+\alpha)/p}}
\end{equation*}
for some $R>r$. Taking $\limsup_{|a|\to1^-}$ and using the Carleson condition again to compare
$\mu(D(a,R))$ and $\mu(D(a,r))$ up to constants yields the desired upper bound.
\end{proof}

Finally, we estimate the essential norm of $T_{g,n}$ on weighted Bergman spaces.

\begin{theorem}\label{thm:bergman-essential-norm}
Let $n\in\mathbb N$, $1\le p\le q<\infty$, and $\alpha,\beta>-1$ satisfy
\begin{equation*}
\frac{q}{p}\le \frac{\beta+2}{\alpha+2}.
\end{equation*}
Let $g$ be analytic in $\mathbb D$, assume that $T_{g,n}:A_\alpha^p\to A_\beta^q$ is bounded Define
\begin{equation*}
d\mu_n(z)=|g'(z)|^q(1-|z|^2)^{\beta+nq}\,dA(z).
\end{equation*}
Then
\begin{equation}\label{eq:bergman-essential}
\left(\|T_{g,n}\|_{e}^{A_\alpha^p\to A_\alpha^q}\right)^{q}
\ \asymp_{\alpha,\beta,p,q,n,r}\
\limsup_{|a|\to 1^-}
\frac{\mu_n(D(a,r))}{(1-|a|^2)^{q(2+\alpha)/p}}.
\end{equation}
In particular, $T_{g,n}:A_\alpha^p\to A_\beta^q$ is compact if and only if the right-hand side of
\eqref{eq:bergman-essential} equals $0$.
\end{theorem}

\begin{proof}
By \eqref{eq:bergman-sobolev} and Proposition~\ref{prop:diff-id},
\begin{equation*}
\|T_{g,n}f\|_{A_\beta^q}^q\asymp_{\beta,q,n}\int_{\mathbb D}|f(z)|^q\,d\mu_n(z),
\quad f\in A_\alpha^p.
\end{equation*}
Applying this to $f=k_a$ and taking $\limsup_{|a|\to 1^-}$ gives
\begin{equation*}
\limsup_{|a|\to 1^-}\|T_{g,n}(k_a)\|_{A_\beta^q}^q
\asymp_{\alpha,\beta,p,q,n}
\limsup_{|a|\to 1^-}\int_{\mathbb D}|k_a(z)|^q\,d\mu_n(z).
\end{equation*}
Using Lemma~\ref{lem:ess-kernel-bergman} (with target $A_\beta^q$) yields
\begin{equation*}
\left(\|T_{g,n}\|_{e}^{A_\alpha^p\to A_\alpha^q}\right)^{\,q}
\asymp_{\alpha,\beta,p,q,n}
\limsup_{|a|\to 1^-}\int_{\mathbb D}|k_a(z)|^q\,d\mu_n(z).
\end{equation*}

Since $T_{g,n}$ is bounded, Theorem~\ref{thm:bergman-boundedness} implies that $\mu_n$ is a $(p,q)$-Carleson measure. Therefore, Lemma~\ref{lem:pq-density-kernel} applies and yields \eqref{eq:bergman-essential}.
\end{proof}

%========================================================
\medskip
\noindent
\textbf{Data Availability.} Throughout this study, no datasets were generated or analyzed.

\medskip
\noindent
\textbf{Competing Interests.} The author declares that he has no conflicting interests.

\medskip
\noindent
\textbf{Founding.} This project has not received any funding.

%=====================

\end{document}